\documentclass{amsart}
\usepackage{enumerate}
\usepackage{ifthen}

\makeindex

\theoremstyle{plain}
 \newtheorem{Theorem}{Theorem}

 \newtheorem{Lemma}{Lemma}
 \newtheorem{Corollary}{Corollary}
\newtheorem{Proposition}{Proposition}

\theoremstyle{definition}
\newtheorem{assumption}{Assumption}

 \newcommand{\N}{\mathbb N}
 \newcommand{\T}{\mathbb T}
 
 \newcommand{\R}{\mathbb R}

\newcommand{\cA}{\mathcal A}
 \newcommand{\cC}{\mathcal C}

  \newcommand{\cL}{\mathcal L}

\newcommand{\cP}{\mathcal P}

\newcommand{\hu}{\hat{u}}
 
\newcommand{\bH}{\bar{H}}
\newcommand{\ub}{\bar{u}}
\newcommand{\bz}{\bar{z}}
\newcommand{\bw}{\bar{w}}
\newcommand{\bL}{\bar{L}}

 \newcommand{\fui}{\varphi}
 \newcommand{\ep}{\varepsilon}

 \newcommand{\de}{\delta}
 
  \newcommand{\lam}{\lambda}
\newcommand{\Lam}{\Lambda}
 
 \newcommand{\si}{\sigma}

 \newcommand{\rip}{\rangle}
 \newcommand{\lip}{\langle}

\newcommand{\diver}{\operatorname{div}}

\newcommand{\leb}{\operatorname{Leb}}

\newcommand{\tr}{\operatorname{Tr}}

\begin{document}
\title{Time-periodic Evans approach to weak KAM theory} 
\author{H\'ector S\'anchez Morgado}
\address{Instituto de Matem\'aticas.Universidad Nacional Aut\'onoma de M\'exico.
Cd. de M\'exico C. P. 04510, M\'exico}
\begin{abstract}
We study the time-periodic version of Evans 
approach to weak KAM theory.
Evans minimization problem is equivalent to a first oder mean field game system.
For the mechanical Hamiltonian  we prove the existence of smooth solutions. 
We introduce the corresponding effective Lagrangian and Hamiltonian and
prove that they are smooth. 
We also consider the limiting behavior of the effective 
Lagrangian and Hamiltonian, Mather measures and minimizers. 
\end{abstract}
\maketitle

\section{Introduction}
\label{sec:intro}

We consider the extension, to time-periodic Hamiltonians
$H:\T^{d+1}\times\R^d\to\R$, of Evans approach to weak KAM theory. 
For $k\in\N$, we address the problem of minimizing 
\begin{equation}
  \label{eq:min-per}
I_k[u]:=\int_{\T^{d+1}}e^{k(u_t +H(x,t,\nabla u))}
\end{equation}
among functions $u:\T^{d+1}\to\R$ with $\int u=0$.
We assume that $H:\T^{d+1}\times\R^d\to\R$ is a smooth function 
strictly convex and superlinear 
i.e. we assume that $H_{pp}$ is positive definite and
\[\frac{H(z,p)}{|p|}\to +\infty, \text{ as } |p|\to +\infty.\]
This variational problem is equivalent to the
first oder mean field game system
\begin{equation}\tag{MFG}\begin{cases}
u_t +H(x,t, \nabla u)=\frac 1k\ln m+\bH_k \\
m_t+\diver(H_pm)=0\\
\int u=0,  \int m=1 
\end{cases} 
\end{equation}
The minimization problem is 
also related through duality to an entropy penalized extension of 
Mather problem (section \ref{sec:epmt}). 
P. Cardaliaguet has also studied  first oder mean field games using 
variational principles in duality \cite{C}, \cite{CG} (with J. Graber), 
as well as other approaches \cite{C1}.
Althought the mean field game we are considering is time dependent, 
the fact that we are searching for time-periodic solutions $(u,m)$
together with the  constant $\bH_k$,
makes more appropriate to consider it as an ergodic problem. The main difference 
with the ergodic problems that have been studied is that the new
Hamiltonian $r+H(z,p)$ is neither coercive nor strictly convex in the variable $(p,r)$. 


According to \cite{CIS} there exists a unique $\bH\in\R$ such that the
the Hamilton Jacobi equation 
\begin{equation}\label{eq:hjt}
 u_t+H(z,\nabla u)=\bH,\quad z=(x,t) 
\end{equation}
has a Lipschitz viscosity solution $\phi:\T^{d+1}\to\R$.  
Moreover $\bH$ is given by the min-max formula
\begin{equation}\label{eq:minmax}
 \bH=\inf_{u\in C^1(\T^{d+1})}\max_{z\in\T^{d+1}} u_t(z)+H(z,\nabla u(z)) 
\end{equation}
which motivates Evans problem of minimizing the functional $I_k$ \cite{E1}.
The convexity of the exponential and Hamiltonian functions imply 
that $I_k$ is lower semicontinuous on $W^{1,q}(\T^{d+1})$ 
but it is not coercive due to the linear term $u_t$.

Using Jensen's inequality and  $\int u_t=0$, we have
\[e^{k\min H}\le\exp\int_{\T^{d+1}} kH(z,\nabla u) \le I_k[u],\]
Letting $e^{k\bH_k}=\inf_u I_k[u]$, we have
\[e^{k\min H} \le e^{k\bH_k}\le I_k[0] \le e^{\max\{H(z,0):z\in\T^{d+1}\}}\]
and thus
\[\min H\le\bH_k\le \max_{z\in\T^{d+1}}H(z,0).\]
A minimizer  must satisfy the Euler Lagrange equation
\begin{equation}\tag{EL}
  (e^{k(u_t+H(z,\nabla u))})_t+\diver(e^{k(u_t+H(z,\nabla u))}
  H_p(z,\nabla u))=0
\end{equation}
which can be written as
\begin{equation}
  \label{eq:pde}
u_{tt}+2H_p\nabla u_t+\nabla^2u(H_p,H_p)+\frac 1k\tr(H_{pp}\nabla^2 u)
+H_t+H_x\cdot H_p+\frac 1k\tr H_{px}=0
\end{equation}
where all derivatives of $H$ are evaluated at $(z,\nabla u(z))$.

Letting $m=e^{k(u_t +H(z,\nabla u)-\bH_k)}$, we transform (EL) 
together with the condition $\int u=0$ and the definition 
of $\bH_k$ into the mean field game system (MFG).

By the convexity of the exponential function
\[\int_{\T^{d+1}} e^{k(u_t+H(x,t,\nabla u))}\ge\int_{\T^{d}} e^{\int_0^1kH(x,t,\nabla u)dt}.\]
If $H$ does not depend on $t$, 
we define $\displaystyle\ub(x)=\int_0^1u(x,t)dt$, then
$\nabla\ub(x)=\displaystyle\int_0^1\nabla u(x,t)dt$. 
From the convexity of $H$ 
\begin{align}
  \int_0^1kH(x,\nabla u(x,t))dt &\ge kH(x,\nabla\ub(x))\\
\int_{\T^{d}} e^{\int_0^1 kH(x,\nabla u)dt}&\ge\int_{\T^{d}} e^{kH(x,\nabla\ub)}
\ge e^{k\bH_{k,a}}
\end{align}
where $e^{k\bH_{k,a}}$ is the minimum for the autonomous problem.
If $v$ is a minimizer for that problem, $v_t=0$ and
\[\int_{\T^{d+1}} e^{kH(x,\nabla v)}=e^{k\bH_{k,a}},\]
so $v$ is also a minimizer for the time dependent problem and $\bH_k=\bH_{k,a}$.

\section{Existence of classical solutions}
\label{sec:exist-solut-eqref}
Writing $z=(x,t)$, $q=(p,r)$, $Du(z)=(\nabla u(z),u_t(z))$,
\eqref{eq:pde} can be written in the form
\begin{equation}\label{eq:spde}
\tr(a_k(z, Du)D^2u)+b_k(z, Du)=0,
 \end{equation}
where
\[a_k(z,q)=
  \begin{pmatrix}
    \frac 1kH_{pp}+H_p\otimes H_p & H_p\\H_p^T & 1
  \end{pmatrix}
=\si\si^T,\quad
\si=\begin{pmatrix}
  (\frac 1kH_{pp})^\frac 12 & H_p \\ 0 & 1
\end{pmatrix}\]
\[b_k(z,q)=H_t+H_x\cdot H_p+\frac 1k\tr H_{px}.\]
The main difficulty to establish the existence of classical solutions of 
\eqref{eq:spde} is that the maximal and minimal eigenvalues of $a_k(z,q)$ go as 
$H_p^2$ and $1/H_p^2$ when $|p|$ tends to $\infty$.

One can obtain apriori Lipschitz bounds for solutions in
particular cases that include the most typical examples of
Hamiltonians. Those are the cases when $b_k$ is sublinear. 
More precisely,
\begin{assumption}\label{A1}
There exists a continuous function $\chi:[0,\infty)\to\R$ such that
\[\int^\infty\frac{ds}{\chi(s)+1}=\infty,\quad |b_k(z,q)|\le\chi(|q|),(z,q)\in
\T^{d+1}\times\R^{d+1} \] 
\end{assumption}
We observe that the assumption holds when
\begin{equation}
  \label{eq:mechanical}
  H(x,t,p)=\frac 12 |p+\eta(t)|^2+V(x,t).
\end{equation}
In fact,
\begin{align*}
b_k(z,p)&=(p+\eta(x,t))\cdot\eta'(t)+\nabla V(x,t)(p+\eta(t))
\end{align*}
does not depend on $k$, and we can take $\chi(s)=cs+d$ for some $c,d>0$. 
\begin{Lemma}\label{unilip}
Under assumption \ref{A1},  there exists $K>0$ depending only on $\chi$ 
such that for a solution $\phi\in  C^2(\T^{d+1})$ of \eqref{eq:spde},
we have $\|D\phi\|_\infty\le K$.
\end{Lemma}
\begin{proof}
For a suitable increasing concave function $\psi:[0,\infty)\to\R$
  with $\psi(0)=0$, consider $h:\R^{d+1}\times\R^{d+1}\to\R$
\[h(z,w)=\phi(z)-\phi(w)-\psi(|z-w|).\]
Since $\phi$ is periodic and $\psi$ is increasing, $h$ achieves its
global maximum in the the cube $[0,1]^{d+1}\times[0,1]^{d+1}$.
So we build $\psi$ on $[0,2]$ and extend it by keeping it increasing 
and concave.
Using Assumption \ref{A1}, one can choose $\psi$ to be a solution of
\[2\psi"=-\chi(\psi')-1 \hbox{ on } (0,2),\quad \psi(0)=0\] 
with $\psi'>0$ on $[0,2]$. In fact, for $K$ sufficiently large 
the function 
\[g(t)=\int_t^K\frac{2\,du}{\chi(u)+1}\]
is strictly decreasing and for some $a>0$, $[0,2]\subset g([a,K])$.
Letting $\psi:[0,2]\to\R$ be the primitive of $g^{-1}$ with $\psi(0)=0$,
we have that $\psi$ satisfies the requeriments.

Observe that $\max\limits_{z,w}h(z,w)\ge 0$ and we want to prove that
$\max h=0$ because in that case 
\[\phi(z)-\phi(w)\le\psi(|z-w|)\le K|z-w|,\]
the last inequality being a consequence of the concavity of $\psi$.
Assume by  contradiction that $\max h$ is positive and it is achieved at $(\bz,\bw)$.
Then 
\[(0,0)=Dh(\bz,\bw)=(D\phi(\bz)-\psi'(|\bz-\bw|)q,-D\phi (\bw)+\psi'(|\bz-\bw|)q),
\quad q=\dfrac{\bz-\bw}{|\bz-\bw|}\]
and 
\begin{align*}
  0\ge& D^2h(\bz,\bw)\\
=&\begin{pmatrix} D^2\phi(\bz)&0\\0&-D^2\phi(\bw)\end{pmatrix}- \psi'(|\bz-\bw|)
  \begin{pmatrix}B&-B\\-B&B\end{pmatrix}-\psi"(|\bz-\bw|)
  \begin{pmatrix}q\otimes q & -q\otimes q \\-q\otimes q &q\otimes q \end{pmatrix},
\end{align*}
where $B=\dfrac{I -q\otimes q}{|\bz-\bw|}$. Thus, $Bq=0$ and then
\[q^T(D^2\phi(\bz)-D^2\phi(\bw))q=(q^T\; -q^T)\begin{pmatrix} D^2\phi(\bz)&0\\0&-D^2\phi(\bw)\end{pmatrix}
  \begin{pmatrix}  q\\-q \end{pmatrix}\le 4\psi"(|\bz-\bw|).\]
Taking $v=\si^{-1}q$ we have
\begin{equation}
  \label{eq:uno}
  (\si v)^TD^2\phi(\bz)\si v-(\si v)^TD^2\phi(\bw)\si v\le 4\psi"(|\bz-\bw|).
\end{equation}

Similarly, for any  $r\in\R^{d+1}$ we have
\begin{equation}
  \label{eq:todos}
  (\si r)^TD^2\phi(\bz)\si r-(\si r)^TD^2\phi(\bw)\si r=((\si r)^T\; (\si r)^T) 
\begin{pmatrix} D^2\phi(\bz)&0\\0&-D^2\phi(\bw)\end{pmatrix}
  \begin{pmatrix}  \si r\\\si r \end{pmatrix}\le 0.
\end{equation}
Inequalities \eqref{eq:uno} and \eqref{eq:todos} imply that
\[\tr\si^T D^2\phi(\bz)\si-\tr\si^T D^2\phi(\bw)\si\le 4\psi"(|\bz-\bw|).\]
Since $\phi$ is a solution of \eqref{eq:spde}, we have
\begin{align*}
  \tr\si^T D^2\phi(\bz)\si+b_k(\bz, \psi'(|\bz-\bw|)q)&=0
\\\tr\si^T D^2\phi(\bw)\si+b_k(\bw, \psi'(|\bz-\bw|)q)&=0
\end{align*}
Thus
\[4\psi"(|\bz-\bw|)\ge b_k(\bw,\psi'(|\bz-\bw|)q)-b_k(\bz,\psi'(|\bz-\bw|)q)
\ge-2\chi(\psi'(|\bz-\bw|)),\]
giving $-2\ge 0$.
\end{proof}
\begin{Theorem}\label{TEU}
For the Hamiltonian \eqref{eq:mechanical}  equation (EL) has a smooth solution 
\end{Theorem}
\begin{proof} We use the continuation method. Consider the family of Hamiltonians 
\begin{equation}
  \label{eq:hlam}
  H_\lam(x,t,p)=\frac 12 |p+\lam\eta(t)|^2+\lam V(x,t)
\end{equation}
and the PDE
\begin{equation}\tag{EL\(_\lam\)}
  (e^{k(u_t+H_\lam(x,t,\nabla u))})_t+\diver(e^{k(u_t+H_\lam(x,t,\nabla u))}
  D_pH_\lam(x,t,\nabla u))=0
\end{equation}
for $u:\T^{d+1}\to\R$ with $\int_{\T^{d+1}}u=0$.

Define 
\[\Lam:=\{\lam\in[0,1]:(\text{EL}_\lam)\text{ has a smooth solution} \}.\]
It is clear that $0\in\Lam$, with $u\equiv 0$. 

We claim that $\Lam$ is closed. In fact, it is clear 
that Assumption \ref{A1} holds  for $H_\lam$ with the same 
$\chi(s)=cs+d$ for all $\lam\in[0,1]$. Therefore Lemma \ref{unilip} implies that there is $K>0$
such that for any $\lam\in\Lam$ the corresponding solution satisfies
$\|Du_\lam\|_\infty\le K$. Elliptic regularity theory implies that we can bound 
uniformly in $\lam\in\Lam$ derivatives of $u_\lam$ of any order.  
Thus, any convergent sequence in $\Lam$ has a subsequence whose
corresponding sequence of solutions converge uniformly, along with all
derivatives.

We claim that $\Lam$ is open. Indeed, for $\lam\in\Lam$ the
linearization of (EL$_\lam$) about the solution $u$ is given by
\begin{multline}\label{linear}
\cL v:=-(e^{k(u_t+H_\lam(z,\nabla u))}(v_t+D_pH_\lam(z,\nabla u)\cdot\nabla v))_t
\\
-\diver(e^{k(u_t+H_\lam(z,\nabla u))}((v_t+D_pH_\lam(z,\nabla u)\cdot\nabla v) D_pH_\lam(z,\nabla u)
+\frac 1k\nabla v D_{pp}H_\lam(z,\nabla u))
\end{multline}
so that
\[\int_{\T^{d+1}} v\cL v=\int_{\T^{d+1}}e^{k(u_t+H_\lam(z,\nabla u))}
Dv^Ta_k(z,Du)Dv\]
and then $\cL$ is a symmetric, uniformly elliptic operator, whose null
space consists of the constants.
The Implicit Function Theorem yields a unique solution for any value
in a neighborhood of $\lam$. 

Since $\Lam$ is nonempty, closed and open it coincides with $[0,1]$.
Thus equation (EL) has a smooth solution.
\end{proof}
\section{Entropy penalized Mather theory}
\label{sec:epmt}

Given a Borel probability $\mu\in\cP(\T^{d+1}\times \R^d)$ we consider 
its push forward $m_{\mu}\in\cP(\T^{d+1})$ given by 
\begin{equation}
\label{push}
\int_{\T^{d+1}} \varphi(z) dm_\mu(z)=\int_{\T^{d+1}\times \R^d} \varphi(z) d\mu(z,v). 
\end{equation}
In the set $\cA\subset\cP(\T^{d+1})$ of measures
absolutely continuous with respect to Lebesgue measure,
the mapping
\[
m\mapsto  S^*[m]:=\int_{\T^{d+1}}\log m(z) dm(z)
\]
is convex and lower semicontinuous. This mapping can be extended in a unique way 
as a convex lower semicontinuous functional to $\cP(\T^{d+1})$ by
\[
\bar S[m]=\liminf_{m_n\in \cA, m_n\rightharpoonup m} S^*[m_n].
\]
Note that the map $\bar S$ is allowed to take the value $+\infty$. 
Furthermore, since $y\ln y\geq -1/e$ we have $\bar S\geq - 1/e$. 
Finally define $ S[\mu]=\bar S[m_\mu]$. 

Let $L:\T^{d+1}\times \R^d \to \R$ be a $C^2$ function, strictly convex 
and superlinear in $v$ 
\[\frac{L(z,v)}{|v|}\to +\infty, \text{ as } |v|\to +\infty.
\]

We consider the convex lower semicontinuous functional 
\begin{equation}
  \label{eq:functional}
  A_{L,k}(\mu)=\int_{\T^d\times \R^d} L(z,v)d\mu+\frac 1k S[\mu], 
\end{equation}
defined  on the space $\cC$ of measures $\mu\in\cP(\T^{d+1}\times\R^d)$
such that 
\begin{enumerate}[(a)]\label{holo}
\item $\displaystyle\int_{\T^{d+1}\times\R^d} |v| d\mu(z,v)<+\infty$
\item For all $\fui\in C^1(\T^{d+1})$,
  $\displaystyle\int_{\T^{d+1}\times \R^d} (\fui_t +\nabla\fui\cdot v)d\mu(x,t,v) =0$. 
\end{enumerate}
From Mather theory we know that for any $Q\in\R^d$ there is $\mu\in\cC$
such that $Q=\int\limits_{\T^{d+1}\times\R^d} v\, d\mu(z,v)$.
We define the effective Lagrangian $\bL_k:\R^d\to\R$  by
\begin{equation}\label{eq:efecL}
  \bL_k(Q):=\inf\{A_{L,k}(\mu):\mu\in\cC,  Q=\int_{\T^{d+1}\times\R^d} v\, d\mu(z,v)\}.
\end{equation}
Let $H:\T ^{d+1}\times\R ^d\to\R$ be Legendre transform of $L$. 
\begin{Proposition}
\begin{equation}
  \label{eq:desigual} 
 \inf_{\mu\in\cC} A_{L,k}(\mu)\ge -\frac 1k \inf_{\fui\in  C^1}\log I_k[\fui].
\end{equation} 
\end{Proposition}
\begin{proof}
Let $\mu\in\cC$ and $\fui\in C^1(\T^{d+1})$. Let $m_\mu$ be given by \eqref{push}. We have
\begin{eqnarray*}
  \int L(z,v) d\mu &=&
\int ( L(z,v) - \fui_t-\nabla\fui\cdot v) d\mu\\ 
&\ge& - \int (\fui_t+H(z,\nabla\fui)) dm_\mu,
\end{eqnarray*}
since, for all $(z,v)\in\T^{d+1}\times \R^d$, $L(z,v)-\nabla\fui(z) v\geq -H(z,\nabla\fui(z))$,
with equality only when $v=H_p(z,\nabla\fui(z))$.
Define $h_\fui:\cP(\T^{d+1})\to\R$ by
\[
h_\fui(m)=\int_{\T^{d+1}} (\fui_t+H(z,\nabla\fui)) dm -\frac 1k \bar S[m]
\]
Then
\[ A_{L,k}(\mu) \ge -h_\fui(m_\mu).\] 
Letting $m_\fui(z)=e^{k(\fui_t+H(z,\nabla\fui))}/I_k[\fui]$,
we have $ h_\fui(m_\fui)=\dfrac 1k \log I_k[\fui]$.

The convex function  $t\mapsto t\log t$ has Legendre transform 
$s\mapsto e^{s-1}$. In particular this implies that  $t\log t+1\ge t$, and so, for any
$m\in\cA$ we obtain
\begin{align*}
 h_\fui(m)\le h(m_\fui,\fui)+\int(\fui_t+H(z,\nabla\fui)-\frac 1k\log m_\fui-\frac 1k)(m(x)-m_\fui(x)) dx.
\end{align*}
The convexity and an approximation argument show that in fact the previous inequality
holds for any $m\in\cP(\T^{d+1})$. 
From the definition of $m_\fui$, and since $m$ and $m_\fui$ are
probability measures, the second term on the rhs vanishes and then
\[\frac 1k\log I_k[\fui]=\max_m h_\fui(m).\]
Therefore,
\[A_{L,k}(\mu)\ge -\frac 1k\log I_k[\fui],\]
and \eqref{eq:desigual} follows.
\end{proof}
 \begin{Proposition}
  If (EL) has a smooth solution $u_k$, $m_k=m_{u_k}$ then
\begin{equation}
\label{matherrep}
\mu_k(z,v)=\delta(v-H_p(z,\nabla u_k(z))) m_k(z),
\end{equation}
is a minimizer of \eqref{eq:functional} and $u_k$ is a minimizer of 
\eqref{eq:min-per}.
\end{Proposition}
\begin{proof}
Definition \eqref{matherrep} means that for all continuous 
function $F:\T^{d+1}\times \R^d\to \R$ we have
\[\int_{\T^{d+1}\times\R^d}F(z,v)d\mu_k=
\int_{\T^{d+1}}F(z,H_p(z,\nabla u))dm_k\]
and since $(u_k,m_k)$ is a solution of (MFG), we get that $\mu_k\in\cC$.
Moreover,
\[A_{L,k}(\mu_k)=\int_{\T^{d+1}}[L(z,H_p(z,\nabla u_k))+\frac 1k\log m_k]dm_k
=-h(m_k,u_k)=-\frac 1k\log I_k[u_k],\]
and then
\begin{equation}
  \label{eq:igual}
A_{L,k}(\mu_k)=\inf_{\mu\in\cC} A_{L,k}(\mu)= -\frac 1k \inf_{\fui\in  C^1}\log I_k[\fui]= 
-\frac 1k\log I_k[u_k].
\end{equation} 
\end{proof}
\begin{Lemma}\label{Lconvex}
The effective Lagrangian $\bL_k$ is convex.  
\end{Lemma}
\begin{proof}
Given $\ep>0$, let $\nu_1, \nu_2$ be such that $A_{L,k}(\nu_i)<\bL_k(Q_i)+\ep$ and 
$\int v\,d\nu_i=Q_i$.
Note that, for $0\leq \lambda\leq 1$,  
\begin{align*}
  \int v\, d(\lam\nu_1+(1-\lam) \nu_2))&=
\lam \int v\, d\nu_1+(1-\lam) \int v\, d\nu_2=\lam Q_1+(1-\lam) Q_2,\\
m_{\lam\nu_1+(1-\lam)\nu_2}&= \lam m_{\nu_1}+(1-\lam)m_{\nu_2}.
\end{align*}
Since $\bar S$ is convex, we have that
\begin{align*}
  \bL_k(\lam\nu_1+(1-\lam)\nu_2) \le A_{L,k}(\lam\nu_1+(1-\lam)\nu_2)
\le &\lam A_{L,k}(\nu_1)+(1-\lam)A_{L,k}(\nu_2)\\
=&\lam\bL_\ep(Q_1)+(1-\lam)\bL_\ep(Q_2)+\ep.
\end{align*}
\end{proof}
For $P\in\R^d$ we define
\begin{equation}\label{eq:efecH}
  \bH_k(P):=\frac 1k\inf_\fui\int_{\T^{d+1}}e^{k(\fui_t +H(x,t,P+\nabla\fui))}
\end{equation}
\begin{Corollary}\label{alfa}
  The Legendre transform of $\bL_k$ is $\bH_k$.
\end{Corollary}
\begin{proof}
  Since the  Legendre transform of $L(z,v)-Pv$
is $H(z,P+p)$ and \eqref{eq:desigual} is an equality, we have that
\begin{equation*}\label{eq:alfa}
\sup_QPQ-\bL_\ep(Q)=-\inf_{\mu\in\cC}A_{L-\lip
  P,\cdot\rip,k}(\mu)=\frac 1k\log\inf_\phi\int e^{k(\phi_t+H(z,P+\nabla\phi))}dx.  
\end{equation*}
\end{proof}
\begin{Lemma}\label{Hconve}
For Hamiltonian \eqref{eq:mechanical}, $\bH_k(P)$ is strictly convex. 
Furthermore for each $P$, \eqref{eq:efecH}
admits at most one minimizer, up to the addition of constants. 
\end{Lemma}
\begin{proof}
For $P\in\R^d$ the Hamiltonian 
$H(x,t,p+P)=\frac 12 |p+P+\eta(x,t)|^2+V(x,t)$ is of the same type.
Suppose there are 
$P_0,P_1\in\R^d$ and $0<\lam<1$ such that  
\[\bH_k(\lam P_0+(1-\lam)P_1)= \lam\bH_k (P_0)+ (1-\lam)\bH_k(P_1).\]
Let $f_i\in C^2(\T^{d+1})$ be a solution of (EL) with $u=P_ix+f_i, i=0,1$ 
so that
\[ \bH_k(P_i)=\frac 1k\log\int_{\T^{d+1}} e^{k(f_{it}+H(z,P_i+\nabla f_i))}.\] 
For $\fui=\lam f_0+ (1-\lam)f_1$ we have
\begin{align*}
\fui_t&=\lam f_{1t}+(1-\lam)f_{2t},\\ 
\nabla\fui+\lam P_0+ (1-\lam)P_1&=\lam (\nabla f_0+P_0)+
                                    (1-\lam)(\nabla f_1+P_1), 
\end{align*}
and, by convexity of $H$,
\begin{equation}
  \label{eq:des}
H(z,\lam P_0+(1-\lam)P_1+\nabla\fui)\le\lam
H(z,\nabla f_0+P_0)+(1-\lam)(H(z,\nabla f_1+ P_1)).
\end{equation}
 
Convexity of the exponential function and H\"older inequality yield
\begin{align}
\notag  e^{k\bH_k(\lambda P_0+(1-\lambda)P_1)} &\leq
\int_{\T^{d+1}} e^{k(\fui_t+H(z,\lam P_0 +(1-\lam)P_1+\nabla\fui))}\\
\notag & \le
\int_{\T^{d+1}}  e^{k(\lam(f_{0t}+H(z,\nabla f_0+P_1))+(1-\lam)(f_{1t}+H(z,\nabla f_1+P_1))}\\ 
\notag &\le\left[\int_{\T^{d+1}}  e^{k(f_{0t}+H(z,\nabla f_0+P_0))}\right]^\lam 
\left[\int_{\T^{d+1}}  e^{k(f_{1t}+H(z,\nabla f_1+P_1))}\right]^{(1-\lam)} \\ 
\label{des} &=e^{\lam k\bH_k(P_0)} e^{(1-\lam)k\bH_k(P_1)}=e^{k\bH_k(\lam P_0 +(1-\lam)P_1)}. 
\end{align}
Therefore all inequalities in \eqref{des} are
equalities and so is \eqref{eq:des}. Since $H$ is strictly convex
$\nabla f_0+P_0= \nabla f_1+P_1$ at all points. Hence $P_1-P_0=\nabla_k(f_0-f_1)$,
and so $P_0=P_1$. Thus $f_0, f_1$ are solutions of \eqref{eq:pde} 
for the same Hamiltonian with $\nabla f_0=\nabla f_1$. Thus $f_{0tt}=f_{1tt}$
and then $f_0-f_1$ is constant.
\end{proof}
\begin{Theorem}
  For the Hamiltonian given by \eqref{eq:mechanical} the effective functions 
$\bL_k, \bH_k$ are smooth.
\end{Theorem}
\begin{proof}
For $P\in\R^d$ consider equation (EL) with $u=Px+\phi$ and define 
$F(P,\phi)$ as the l.h.s. of that equation. We have seen that 
for  a solution $\phi=\phi(\cdot,P)$ of $F(P,\phi)=0$, $\cL=D_2F(P,\phi)$ 
is given by \eqref{linear}. The Implicit Function Theorem implies 
that $\phi(\cdot,P)$ is smooth in $P$ and so $\bH_k$ are smooth. 
Moreover $\bH_k$ is stricltly convex so $D\bH_k$ has a smooth inverse 
$G_k$ and therefore $\bL_k(Q)=QG_k(Q)-\bH_k(G_k(Q))$ is smooth.
\end{proof}
\section{Approximating weak KAM theory}
\label{sec:converge}

In this section we assume the Hamiltonian is given 
by \eqref{eq:mechanical} so that  $b_k$ satisfies assumption \eqref{A1}
with $\chi$ independent of $k$. 
Let $u_k$ be the minimizer of \eqref{eq:min-per} with $\int u_k=0$.
From Lemma \ref{unilip} there is $K$ such that $\|Du_k\|_\infty\le K$ for any 
$k$, so passing to a subsequence, $u_k$ converges uniformly to a Lipschitz 
$u:\T^{d+1}\to\R$ and 
$Du_k\rightharpoonup Du$ weakly in $L^q(\T^{d+1},\R^{d+1})$, for any $1\le q<\infty$.

As in the autonomous case we have the following Theorem which has a
similar proof
\begin{Theorem} \label{awkam}\quad
  \begin{enumerate}[(i)]
  \item We have $\lim\limits_{k\to\infty}\bH_k=\bH.$
  \item Function $u$ is a viscosity solution of
 \begin{equation}\label{aronsson}
u_{tt}+2H_p\nabla u_t+\nabla^2u(H_p,H_p)+H_t+H_x\cdot H_p=0
\end{equation} 
\item Moreover, $u_t+H(z,\nabla u)\le\bH$ Lebesgue a.e. in $\T^{d+1}$. 
  \end{enumerate}
\end{Theorem}

  \begin{Proposition}\label{matherlimit}
Let $\mu_k$ be the measure defined by \eqref{matherrep}.
Passing to a subsequence such that $\mu_k\rightharpoonup\mu$, we have  
\begin{enumerate}[(a)]
\item $\mu$ is a Mather measure 
\item $\lim\limits_{k\to\infty}\dfrac 1k S[\mu_k]=0$
\end{enumerate}
\end{Proposition}
\begin{proof}
The function $\dfrac 1k\log m_k=u_{kt}+H(z,\nabla u_k)$ is uniformly bounded. 
For $\lam>0$,
\begin{align*}
\int_{\T^{d+1}}\frac 1k\log m_kdm_k&
=\int\limits_{\{\log m_k\ge-k\lam\}}\frac 1k\log m_kdm_k+
\int\limits_{\{\log m_k<-k\lam\}}\frac 1k\log m_kdm_k  \\
&\ge -\lam\int\limits_{\{\log m_k\ge-k\lam\}}dm_k-
\int\limits_{\{\log_k<-k\lam\}} Ce^{-k\lam} dz\\
&\ge -\lam+Ce^{-k\lam}.
\end{align*}
Thus,
\[\liminf_{k\to\infty} \frac 1kS[\mu_k]\ge-\lam\]
Since $\lam>0$ is arbitrary
\begin{equation}
  \label{eq:min-seq}
\liminf_{k\to\infty}\frac 1kS[\mu_k]\ge 0.
\end{equation}
We recall that 
\[-\bH = \min_{\mu\in\cC}\int L\,d\mu.\]
From \eqref{eq:min-seq}
\begin{align*}
\int L\, d\mu&=\lim_{k\to\infty}\int L\, d\mu_k
=\lim_{k\to\infty}A_{L,k}(\mu_k)-\frac 1k S[\mu_k]\\
&=-\lim_{k\to\infty}\bH_k-\frac 1k S[\mu_k]\le -\bH.
\end{align*}
Therefore, $\mu$ is a minimizing measure, the inequality is an equality
and (b) holds. 
\end{proof}

\end{document}